\documentclass[11pt]{article}

\usepackage{amssymb,amsmath,amsthm}

\newcommand{\ds}{\displaystyle}

\newcommand{\lam}{\lambda}
\newcommand{\tet}{\theta}
\newcommand{\abs}[1]{\left|#1\right|}

\newcommand{\bS}{\mathbb{S}}

\newtheorem{thm}{Theorem}
\newtheorem{lem}[thm]{Lemma}
\newtheorem{cor}[thm]{Corollary}

\begin{document}

\begin{center}
{\large On a function related to Chowla's cosine problem} \\[1ex]
Idris Mercer, University of Delaware, \texttt{idmercer@math.udel.edu}
\end{center}

{\bf Abstract:} There is a rich literature that concerns the minimum value
of expressions of the form
$\cos a_1\tet + \cdots + \cos a_n\tet$
(where $a_1 < \cdots < a_n$ are positive integers)
and the question of what choice of $\{a_1,\ldots,a_n\}$
maximizes that minimum.
A related problem concerns the maximum minimum
(on the unit circle) of expressions of the form
$\abs{z^{a_1}+\cdots+z^{a_n}}$.
If we define
\begin{align*}
\lam(n) & = - \sup_{a_1<\cdots<a_n}
\min_{\tet} (\cos a_1\tet + \cdots + \cos a_n\tet), \\
\mu(n) & = \sup_{a_1<\cdots<a_n}
\min_{\abs{z}=1} \abs{z^{a_1}+\cdots+z^{a_n}},
\end{align*}
then one can ask either for bounds on the functions $\lam(n)$ or $\mu(n)$,
or particular values of $\lam(n)$ or $\mu(n)$.
Other authors have found the values of $\mu(3)$ and $\mu(4)$.
In this paper, we find the values of $\lam(2)$ and $\lam(3)$. \\[1ex]

\section{Introduction}

We define a {\bf cosine polynomial} of {\bf length}~$n$
to be any expression of the form
$$
\cos a_1\tet + \cos a_2\tet + \cdots + \cos a_n\tet
$$
where $a_1 < \cdots < a_n$ are integers $\ge 1$,
and we define a {\bf Newman polynomial} of {\bf length}~$n$
to be any expression of the form
$$
z^{a_1} + z^{a_2} + \cdots + z^{a_n}
$$
where $a_1 < \cdots < a_n$ are integers $\ge 0$.

We are interested in the minimum value of a length~$n$
cosine polynomial, and the minimum modulus of a length~$n$
Newman polynomial on the unit circle
(we will denote the unit circle by $\bS$).
We want to maximize those minima.

We define
\begin{align*}
L(a_1,\ldots,a_n) & = \min_\tet ( \cos a_1\tet + \cdots + \cos a_n\tet ) \\
M(a_1,\ldots,a_n) & = \min_{z\in\bS} \abs{ z^{a_1} + \cdots + z^{a_n} }
\end{align*}
so $-n \le L(a_1,\ldots,a_n) \le 0$ and $0 \le M(a_1,\ldots,a_n) \le n$.
We then define
\begin{align*}
\lam(n) & = - \sup L(a_1,\ldots,a_n) \\
\mu(n) & = \sup M(a_1,\ldots,a_n)
\end{align*}
where the supremum in the definition of $\lam$ is taken over
all sets of $n$ distinct positive integers, and the supremum
in the definition of $\mu$ is taken over all sets of $n$ distinct
nonnegative integers.
Note that $\lam(n)$ and $\mu(n)$ are both nonnegative.

Both $\lam(n)$ and $\mu(n)$ are mathematically well-defined,
because we are taking the supremum of a bounded set.
However, since there are infinitely many possible sets
$\{a_1,\ldots,a_n\}$, it is not obvious how to find the value
of $\lam(n)$ or $\mu(n)$ for a particular $n$ in a finite number
of steps.

Three types of problems we might consider are:
\begin{enumerate}
\item finding upper or lower bounds for the functions
$\lam(n)$ or $\mu(n)$,
\item finding values of $\lam(n)$ or $\mu(n)$ for particular $n$,
\item showing that one can calculate particular values of
$\lam(n)$ or $\mu(n)$ in a finite number of steps.
\end{enumerate}

Even proving $\lam(n) \to \infty$ is nontrivial.
This was first done by Uchiyama and Uchiyama \cite{UU60}
using results of Cohen \cite{Coh60};
their lower bound for $\lam(n)$ was sublogarithmic.
Over the years, better lower bounds for $\lam(n)$
have been found. The best lower bound currently known
is due to Ruzsa \cite{Ruz04}; it is superlogarithmic
but grows more slowly than any power of~$n$.
The best known upper bound for $\lam(n)$ appears to be
$O(\sqrt{n})$. Chowla conjectured \cite{Cho65}
that this is the true rate of growth.

Less appears to be known about the growth rate of $\mu(n)$.
By considering the $L^2$ norm,
one can show that $\mu(n) \le \sqrt{n}$ for all~$n$,
and by considering a particular length~$9$ Newman polynomial,
one can show that $\mu(n) \ge n^{0.14}$ when $n$ is a power of~$9$.
In~\cite{Boy86},
Boyd considered the maximum minimum modulus on~$\bS$
of Newman polynomials of {\bf degree}~$n$,
but also made some conjectures about the function denoted
by~$\mu(n)$ in this paper. Specifically, he conjectured that
$\mu(n) > 1$ for all $n \ge 6$, and conjectured that
$\log \mu(n)/\log n$ approaches a limit.

Some particular values of $\mu(n)$ have been computed:
Campbell, Ferguson, and Forcade \cite{CFF83} proved that
$$
\mu(3) = \sqrt{\frac{47-14\sqrt{7}}{27}} \approx 0.607346
$$
and Goddard \cite{God92} proved that
$$
\mu(4) = \min_{-1\le x\le1}\sqrt{16x^4+8x^3-8x^2-2x+2} \approx 0.752394.
$$
The current author is unaware of any proofs in the literature
for particular values of~$\lam(n)$. The main results of this paper are proofs
that
\begin{align*}
\lam(2) & = \frac98 = 1.125000 \hspace{2em} \mbox{and} \\[1ex]
\lam(3) & = \frac{17+7\sqrt{7}}{27} \approx 1.315565.
\end{align*}
We note that one can make plausible guesses about
other values of $\lam(n)$ and $\mu(n)$ by searching
cosine polynomials or Newman polynomials of bounded degree.
The conjectured values of $\mu(5)$ and $\mu(6)$ given below
appear in~\cite{God92} and were obtained by considering
$a_n \le 30$. (There is a small
error in~\cite{God92}; the author mistakenly writes the
square of the conjectured value of~$\mu(6)$.)
The conjectured values of $\lam(4)$, $\lam(5)$, and $\lam(6)$
were obtained by the current author by considering $a_n \le 20$.
It was conjectured in~\cite{God92} that $\mu(n)$ is monotone,
but note that if our conjectured values for $\lam(5)$ and $\lam(6)$
are correct, then $\lam(n)$ is not monotone.
Perhaps $\lam(n)$ is eventually monotone.

\begin{center}
\begin{tabular}{c|l|l}
$n$ & Suspected $\lam(n)$ & $a_1,\ldots,a_n$ that attain \\
& & suspected optimal value \\ \hline
2 & 1.125000 (proved) & \hspace{2em} 1,2 \\
3 & 1.315565 (proved) & \hspace{2em} 1,2,3 \\
4 & 1.519558 & \hspace{2em} 1,2,3,4 \\
5 & 1.627461 & \hspace{2em} 1,2,4,5,6 \\
6 & 1.591832 & \hspace{2em} 1,2,4,6,7,8 \\
\end{tabular}
\end{center}

\begin{center}
\begin{tabular}{c|l|l}
$n$ & Suspected $\mu(n)$ & $a_1,\ldots,a_n$ that attain \\
& & suspected optimal value \\ \hline
3 & 0.607346 (proved) & \hspace{2em} 0,1,3 \\
4 & 0.752394 (proved) & \hspace{2em} 0,1,2,4 \\
5 & 1.000000 & \hspace{2em} 0,1,2,6,9 \\
6 & 1.065286 & \hspace{2em} 0,6,9,10,17,24 \\
\end{tabular}
\end{center}

\section{Results}

\begin{lem} \label{nine-eighths}
The minimum value of $f(\tet) = \cos\tet + \cos2\tet$
is $-9/8$, which occurs when $\cos\tet = -1/4$.
\end{lem}

\begin{proof}
This is an elementary exercise in trigonometric identities and calculus.
\end{proof}

\begin{lem} \label{degree3}
The minimum value of $f(\tet) = \cos\tet + \cos2\tet +\cos3\tet$ is
$$
\frac{-17-7\sqrt{7}}{27} \approx -1.315565,
$$
which occurs when $\cos\tet = (-1+\sqrt{7})/6$.
\end{lem}

\begin{proof}
This is elementary as well, and follows from writing
$$
\cos\tet + \cos2\tet +\cos3\tet = \cos\tet + (2\cos^2\tet-1) + (4\cos^3\tet-3\cos\tet)
$$
and then minimizing $4c^3 + 2c^2 - 2c - 1$ for $-1 \le c \le 1$.
\end{proof}

The rest of this paper consists of showing that the minima
appearing in Lemmas 1 and~2 cannot be improved upon by choosing
other values of the~$a_j$.

Note that it suffices to consider the case $\gcd(a_1,\ldots,a_n) = 1$.
This is because if $d$ divides all $a_j$, then the cosine polynomials
$$
\cos a_1\tet + \cdots + \cos a_n\tet \hspace{2em} \mbox{and} \hspace{2em}
\cos \frac{a_1}{d}\tet + \cdots + \cos \frac{a_n}{d}\tet
$$
attain the same set of values.

\begin{thm} \label{two}
If $f(\tet) = \cos a_1\tet + \cos a_2\tet$ where
$a_1 < a_2$ are relatively prime positive integers
and $a_2 \ge 3$, then we have
$f(\tet) \le -3/2$
for some $\tet$.
\end{thm}

\begin{proof}
If $a_1$ and $a_2$ are both odd then $f(\pi) = -2$,
so assume one of $a_1,a_2$ is odd and the other is even.
Now observe that if
$$
\tet \in S := \bigg\{ \frac{k\pi}{a_2} \;\bigg|\; \mbox{$k$ is an odd integer} \bigg\}
$$
then $\cos a_2\tet = -1$. So it would suffice to prove
that $\cos a_1\tet \le -1/2$ for some $\tet \in S$.

Case 1.
Suppose $a_1$ is odd and $a_2$ is even.
Note that $a_1$ and $2a_2$ are relatively prime,
so we have
$$
a_1 s + 2a_2 t = 1
$$
for some integers $s$ and $t$. Note that $s$ must be odd.
We then have
\begin{align*}
a_1 s & = 1 - 2a_2 t \\
a_1(a_2\!-\!1)s & = (a_2\!-\!1) - 2a_2(a_2\!-\!1)t \\
a_1\frac{(a_2\!-\!1)s\pi}{a_2} & = \frac{a_2\!-\!1}{a_2}\pi - 2(a_2\!-\!1)t\pi \\
a_1\frac{(a_2\!-\!1)s\pi}{a_2} & = \frac{a_2\!-\!1}{a_2}\pi
- \mbox{integer multiple of $2\pi$} \\
\cos\bigg( a_1\frac{(a_2\!-\!1)s\pi}{a_2} \bigg)
& = \cos\bigg( \frac{a_2\!-\!1}{a_2}\pi \bigg)
\end{align*}
Now note that $a_2\!-\!1$ and $s$ are both odd, so $(a_2\!-\!1)s\pi/a_2 \in S$.
Note also that $a_2 \ge 3$ implies
\begin{align*}
\frac{2\pi}{3} & \le \frac{a_2\!-\!1}{a_2}\pi < \pi \\
-\frac12 & \ge \cos\bigg( \frac{a_2\!-\!1}{a_2}\pi \bigg) > -1
\end{align*}
so Case 1 is complete.

Case 2.
Suppose $a_1$ is even and $a_2$ is odd.
Since $a_1$ and $a_2$ are relatively prime,
we can write
$$
a_1 s + a_2 t = -1
$$
for some integers $s$ and $t$. Note that $t$ must be odd.
We then have
\begin{align*}
a_1 s + a_2 (t\!+\!1) & = a_2\!-\!1 \\
a_1 (s\!+\!a_2) + a_2 (t\!+\!1\!-\!a_1) & = a_2\!-\!1
\end{align*}
Note that since $a_2$ is odd, one of $s$ and $s\!+\!a_2$
must be odd.
If $s$ is odd, define
\begin{align*}
s' & = s \\
t' & = t\!+\!1
\end{align*}
and if $s\!+\!a_2$ is odd, define
\begin{align*}
s' & = s\!+\!a_2 \\
t' & = t\!+\!1\!-\!a_1
\end{align*}
Then $s'$ is odd, $t'$ is even, and we have
\begin{align*}
a_1 s' + a_2 t' & = a_2\!-\!1 \\
a_1 \frac{s'\pi}{a_2} + t'\pi & = \frac{a_2\!-\!1}{a_2}\pi \\
a_1 \frac{s'\pi}{a_2} & = \frac{a_2\!-\!1}{a_2}\pi - \mbox{integer multiple of $2\pi$} \\
\cos\bigg( a_1 \frac{s'\pi}{a_2} \bigg)
& = \cos\bigg( \frac{a_2\!-\!1}{a_2}\pi \bigg)
\end{align*}
Note that $s'\pi/a_2 \in S$, and as before, $\cos((a_2\!-\!1)\pi/a_2) \le -1/2$.
So Case 2 is complete.
\end{proof}

The following is a consequence of Lemma~\ref{nine-eighths}
and Theorem~\ref{two}.

\begin{cor}
We have $\lam(2) = 9/8 = 1.125$.
\end{cor}

Next, we give some lemmas
that will be helpful for evaluating $\lam(3)$.

\begin{lem} \label{cosinesum}
Suppose $\xi$ is a real number, $k \ge 2$ is an integer,
and $m$ is an integer. If $m$ is not a multiple of~$k$, then
$$
\sum_{j=0}^{k-1} \cos\Big( \xi + \frac{2\pi m j}{k} \Big) = 0.
$$
\end{lem}

\begin{proof}
This follows from the fact that the above sum is the real part of
$$
\sum_{j=0}^{k-1} \exp \bigg( i \Big( \xi + \frac{2\pi m j}{k} \Big) \bigg).
$$
\end{proof}

\begin{lem} \label{variance}
Let $y_0, \ldots, y_{N-1}$ be real numbers satisfying
$\sum_{j=0}^{N-1} y_j = 0$.
Suppose $M > 0$ is a real number such that
$y_j \le M$ for all~$j$
(so not all $y_j$ are equal to~$M$).
If we have
$$
\frac1N \sum_{j=0}^{N-1} y_j^2 \ge KM
$$
for some real number $K$
(we will take $K>0$), then we have $y_j \le -K$ for some~$j$.
\end{lem}

\begin{proof}
We use contraposition. Suppose $y_j > -K$ for all~$j$.
Then $y_j + K$ is always positive. Now note that $M - y_j$
is always nonnegative, and is sometimes strictly positive.
We therefore have
\begin{align*}
\frac1N \sum_{j=0}^{N-1} (M - y_j)(y_j + K) & > 0 \\
\frac1N \sum_{j=0}^{N-1} \Big(-y_j^2 +(M-K)y_j +KM\Big) & > 0 \\
-\frac1N \sum_{j=0}^{N-1} y_j^2 + (M-K)0 + KM & > 0 \\
KM & > \frac1N \sum_{j=0}^{N-1} y_j^2.
\end{align*}
\end{proof}

\begin{thm} \label{three}
Suppose $f(\tet) = \cos a_1\tet + \cos a_2\tet + \cos a_3\tet$
where $a_1 < a_2 < a_3$ are positive integers satisfying
$\gcd(a_1,a_2,a_3)=1$. Then for some $\tet$, we have
$$
f(\tet) \le \frac{-17-7\sqrt{7}}{27} \approx -1.315565.
$$
\end{thm}

\begin{proof}
We split the proof into three cases:
\begin{itemize}
\item Case 1: $a_3 = a_1 + a_2$,
\item Case 2: $a_3 = 2a_1$ or $a_3 = 2a_2$,
\item Case 3: $a_3 \notin \{ a_1+a_2, 2a_1, 2a_2 \}$.
\end{itemize}

Case 1.
Note that if $z = e^{i\tet} \in \bS$, we have
\begin{align*}
\abs{1+z^{a_1}+z^{a_1+a_2}}^2
& = (1+z^{a_1}+z^{a_1+a_2})(1+z^{-a_1}+z^{-a_1-a_2}) \\
& = 3 + 2\big(\cos a_1\tet + \cos a_2\tet + \cos(a_1\!+\!a_2)\tet\big).
\end{align*}
Since $1+z^{a_1}+z^{a_1+a_2}$ is a Newman polynomial
of length~$3$, we have
\begin{align*}
\abs{1+z^{a_1}+z^{a_1+a_2}} \le \mu(3) = \sqrt{\frac{47-14\sqrt{7}}{27}}
\end{align*}
for some $z = e^{i\tet} \in \bS$, by Theorem 2 in \cite{CFF83}.
Therefore for some $\tet$, we have
\begin{align*}
3 + 2\big(\cos a_1\tet + \cos a_2\tet + \cos(a_1\!+\!a_2)\tet\big)
& \le \frac{47-14\sqrt{7}}{27} \\[1ex]
2\big(\cos a_1\tet + \cos a_2\tet + \cos(a_1\!+\!a_2)\tet\big)
& \le \frac{-34-14\sqrt{7}}{27} \\[1ex]
\cos a_1\tet + \cos a_2\tet + \cos(a_1\!+\!a_2)\tet
& \le \frac{-17-7\sqrt{7}}{27}.
\end{align*}
This completes Case 1.

Case 2.
If $a_3=2a_1$, define
$$
a=a_1, \hspace{2em} b=a_2, \hspace{2em} c=a_3
$$
and if $a_3 = 2a_2$, define
$$
a=a_2, \hspace{2em} b=a_1, \hspace{2em} c=a_3.
$$
So we have
$$
f(\tet) = \cos a\tet + \cos b\tet + \cos 2a\tet
$$
where either
\begin{equation} \label{case2ineqs}
1 \le a < b < 2a \hspace{2em} \mbox{or} \hspace{2em} 1 \le b < a < 2a.
\end{equation}
So $a \ge 2$.

If $a=2$, the only possibilities for $f$ are
\begin{align*}
f(\tet) & = \cos 1\tet + \cos 2\tet + \cos 4\tet, \\
f(\tet) & = \cos 2\tet + \cos 3\tet + \cos 4\tet.
\end{align*}
We dispose of those possibilities by observing
\begin{align*}
\cos \Big(1\cdot\frac{2\pi}{3}\Big) + \cos \Big(2\cdot\frac{2\pi}{3}\Big)
+ \cos \Big(4\cdot\frac{2\pi}{3}\Big) & = -\frac32 < -1.315565, \\[1ex]
\cos \Big(2\cdot\frac{\pi}{3}\Big) + \cos \Big(3\cdot\frac{\pi}{3}\Big)
+ \cos \Big(4\cdot\frac{\pi}{3}\Big) & = -2 < -1.315565.
\end{align*}
So for the rest of Case 2, we assume $a > 2$.

We note from Lemma \ref{nine-eighths}
that $\cos a\tet + \cos 2a\tet$ attains its
minimum value of $-9/8$ when $\cos a\tet = -1/4$. Define
$$
\xi = \arccos\Big(-\frac14\Big) \approx 1.823477
$$
and further define
$$
\tet_j = \frac{\xi}{a} + \frac{2\pi j}{a} \hspace{2em}
\mbox{for $j = 0,1,\ldots,a-1$.}
$$
Then $\cos a\tet_j + \cos 2a\tet_j = -9/8$ for each $j$.
We claim that $\cos b\tet_j \le -1/2$ for some~$j$,
implying that $f(\tet_j) \le -13/8 = -1.625$,
which will take care of Case~2.

We will use Lemma \ref{variance}. We choose
$$
y_j = \cos b\tet_j \hspace{2em} \mbox{for $j=0,1,\ldots,a-1$}.
$$
We need to prove $\sum_j y_j =0$. We can take $M=1$,
and we will show that $\frac1a \sum_j y_j^2 = \frac12$.
The claim will then follow.

Note that
$$
\sum_{j=0}^{a-1} y_j
= \sum_{j=0}^{a-1} \cos \Big( \frac{b\xi}{a} + \frac{2\pi b j}{a} \Big).
$$
This is 0 by Lemma \ref{cosinesum}, because $b$ is not a multiple of~$a$
(since $b < 2a$ and $b \ne a$). Now consider
\begin{align*}
\frac1a \sum_{j=0}^{a-1} y_j^2
& = \frac1a \sum_{j=0}^{a-1}
\cos^2 \Big( \frac{b\xi}{a} + \frac{2\pi b j}{a} \Big) \\[1ex]
& = \frac1a \sum_{j=0}^{a-1}
\bigg( \frac12 + \frac12 \cos\Big( \frac{2b\xi}{a} + \frac{2\pi 2b j}{a} \Big) \bigg) \\[1ex]
& = \frac12
+ \frac1{2a} \sum_{j=0}^{a-1} \cos\Big( \frac{2b\xi}{a} + \frac{2\pi 2b j}{a} \Big)
\end{align*}
which, by Lemma \ref{cosinesum}, is equal to $\frac12 + 0$
if we can show $2b$ is not a multiple of~$a$.
From~(\ref{case2ineqs}), note that we have
$$
2a < 2b < 4a \hspace{2em} \mbox{or} \hspace{2em} 2b < 2a
$$
so if $2b$ is a multiple of~$a$, we have $2b = 3a$ or $2b = a$.
This implies $a$ is even, so say $a = 2k$. Since $a>2$, we have
$k>1$. Note that either $b=3k$ or $b=k$. Thus $k>1$ divides all
elements of $\{a,b,2a\}=\{a_1,a_2,a_3\}$, which is a contradiction.
Therefore, as stated, $2b$ is not a multiple of~$a$.
This completes the verification of the claim and thus completes Case~2.

Case 3.
Assume $a_3 \notin \{a_1+a_2,2a_1,2a_2\}$. We define
$$
\tet_j = \frac{\pi}{a_3} + \frac{2\pi j}{a_3} \hspace{2em}
\mbox{for $j = 0,1,\ldots,a_3-1$}
$$
so $\cos a_3 \tet_j = -1$ for each $j$.
We claim that $\cos a_1 \tet_j + \cos a_2 \tet_j \le -1/2$ for some~$j$,
implying that $f(\tet_j) \le -3/2 = -1.5$, which will take care of Case~3.

We will use Lemma \ref{variance}. We choose
$$
y_j = \cos a_1 \tet_j + \cos a_2 \tet_j \hspace{2em} \mbox{for $j=0,1,\ldots,a_3-1$}.
$$
We need to prove $\sum_j y_j = 0$.
We can take $M=2$, and we will show that
$\frac1{a_3} \sum_j y_j^2 = 1$, so we can take $K=1/2$.
The claim will then follow.

Note that
$$
\sum_{j=0}^{a_3-1} y_j
= \sum_{j=0}^{a_3-1}
\bigg( \cos \Big( \frac{a_1\pi}{a_3} + \frac{2\pi a_1j}{a_3} \Big)
+ \cos \Big( \frac{a_2\pi}{a_3} + \frac{2\pi a_2j}{a_3} \Big) \bigg)
$$
which is 0 by Lemma \ref{cosinesum}, since neither $a_1$
nor $a_2$ is a multiple of $a_3$.
Now consider
\begin{align*}
& \; \frac1{a_3} \sum_{j=0}^{a_3-1} y_j^2 \\[1ex]
= & \; \frac1{a_3} \sum_{j=0}^{a_3-1}
\bigg( \cos \Big( \frac{a_1\pi}{a_3} + \frac{2\pi a_1j}{a_3} \Big)
+ \cos \Big( \frac{a_2\pi}{a_3} + \frac{2\pi a_2j}{a_3} \Big) \bigg)^2 \\[1ex]
= & \; \frac1{a_3} \sum_{j=0}^{a_3-1}
\cos^2 \Big( \frac{a_1\pi}{a_3} + \frac{2\pi a_1j}{a_3} \Big) \\[1ex]
& {} + \frac1{a_3} \sum_{j=0}^{a_3-1}
2\cos \Big( \frac{a_1\pi}{a_3} + \frac{2\pi a_1j}{a_3} \Big)
\cos \Big( \frac{a_2\pi}{a_3} + \frac{2\pi a_2j}{a_3} \Big) \\[1ex]
& {} + \frac1{a_3} \sum_{j=0}^{a_3-1}
\cos^2 \Big( \frac{a_2\pi}{a_3} + \frac{2\pi a_2j}{a_3} \Big) \\[1ex]
= & \; \frac1{a_3} \sum_{j=0}^{a_3-1} \bigg( \frac12 + \frac12
\cos \Big( \frac{2a_1\pi}{a_3} + \frac{2\pi 2a_1j}{a_3} \Big)
\bigg) \\[1ex]
& {} + \frac1{a_3} \sum_{j=0}^{a_3-1} \bigg(
\cos \Big( \frac{(a_2\!-\!a_1)\pi}{a_3}\!+\!\frac{2\pi(a_2\!-\!a_1)j}{a_3} \Big)
+ \cos \Big( \frac{(a_2\!+\!a_1)\pi}{a_3}\!+\!\frac{2\pi(a_2\!+\!a_1)j}{a_3} \Big)
\bigg) \\[1ex]
& {} + \frac1{a_3} \sum_{j=0}^{a_3-1} \bigg( \frac12 + \frac12
\cos \Big( \frac{2a_2\pi}{a_3} + \frac{2\pi 2a_2j}{a_3} \Big)
\bigg).
\end{align*}
By Lemma \ref{cosinesum}, this is equal to $\frac12+\frac12=1$
if we can show that none of the numbers
$2a_1, a_2\!-\!a_1, a_2\!+\!a_1, 2a_2$
is a multiple of~$a_3$. This follows because
\begin{itemize}
\item $0 < 2a_1 < 2a_3$ and $2a_1 \ne a_3$,
\item $0 < a_2\!-\!a_1 < a_3$,
\item $0 < a_2\!+\!a_1 < 2a_3$ and $a_2\!+\!a_1 \ne a_3$,
\item $0 < 2a_2 < 2a_3$ and $2a_2 \ne a_3$.
\end{itemize}
This completes the verification of the claim and thus completes Case~3.
\end{proof}

The following is a consequence of Lemma~\ref{degree3}
and Theorem~\ref{three}.

\begin{cor}
We have $\lam(3) = \ds\frac{17+7\sqrt{7}}{27} \approx 1.315565$.
\end{cor}

We remark in closing that it would be interesting if
the evaluation of $\lam(4)$ or $\lam(5)$ or $\lam(6)$
can somehow be reduced to a finite search (even an impractically
large finite search). Very roughly speaking, cosine polynomials
of large degree have many local minima and do not appear to be
good candidates for high minima.

%\vspace{-0.1in}

{Idris Mercer, Department of Mathematical Sciences,
University of Delaware, Newark DE, 19716,
\texttt{idmercer@math.udel.edu}}

\end{document}